\documentclass{amsart}
\usepackage{geometry,amssymb,color}
\usepackage{enumerate}
\usepackage{color}

\title{Non-local torsion functions and embeddings}
\author[Franzina]{Giovanni Franzina}

\address[G. Franzina]{Istituto Nazionale di Alta Matematica (IN{\rm d}AM)
	\newline\indent
	Unit\`a di Ricerca di Firenze c/o DiMaI ``Ulisse Dini'' Universit\`a di Firenze
	\newline\indent
	Viale Morgagni 67/A, I-50134 Firenze, Italy}
\email{franzina@math.unifi.it}

\numberwithin{equation}{section}

\usepackage{mathtools}
\usepackage[colorlinks=true,urlcolor=blue,
citecolor=red,linkcolor=blue,
linktocpage,pdfpagelabels,bookmarksnumbered,bookmarksopen
]{hyperref}

\newcommand{\R}{\mathbb R}

\def\case#1#2{\par\noindent{{ if~#1,}}{#2}\\}
\def\XXint#1#2#3{{\setbox0=\hbox{$#1{#2#3}{\int}$} \vcenter{\vspace{-1pt}\hbox{$#2#3$}}\kern-.5\wd0}}
\def\Xint#1{\mathchoice {\XXint\displaystyle\textstyle{#1}}{\XXint\textstyle\scriptstyle{#1}}{\XXint\scriptstyle\scriptscriptstyle{#1}}{\XXint\scriptscriptstyle\scriptscriptstyle{#1}}\!\int}
\def\intmed{\Xint{-}}
\newtheorem{lm}{Lemma}[section]
\newtheorem{prop}[lm]{Proposition}
\newtheorem{cor}[lm]{Corollary}
\newtheorem{teo}[lm]{Theorem}

\theoremstyle{definition}

\newtheorem{defi}[lm]{Definition}
\newtheorem{rmq}[lm]{Remark}
\newtheorem*{ack}{Acknowledgments}

\begin{document}
\maketitle
{\quad}
\begin{abstract}
\small{
Given $s\in(0,1)$, we discuss the embedding of $\mathcal{D}_0^{s,p}(\Omega)$
in $L^q(\Omega)$. In particular, for $1\le q<p$ we deduce its compactness
on all open sets $\Omega\subset \R^N$ on which it is continuous.
We then relate, for all $q$ up the fractional Sobolev conjugate exponent,
the continuity of the embedding to the summability of
the function solving
the fractional torsion problem in $\Omega$ 
in a suitable weak sense, for every open set $\Omega$.
The proofs make use of a non-local Hardy-type inequality in $\mathcal{D}_0^{s,p}(\Omega)$,
involving
the fractional torsion function
as a weight.

}

\end{abstract}
\section{Introduction}
Let $\Omega$ be an open set in $\R^N$, let $1<p<\infty$, let $0<s<1$,
and let $\mathcal{D}_0^{s,p}(\Omega)$ be the homogeneous fractional Sobolev space
obtained by completion of $C^\infty_0(\Omega)$ with respect to the Gagliardo norm
\[
	\left(\iint_{\mathbb R^{2N}}\frac{|u(x)-u(y)|^p}{|x-y|^{N+sp}}\,dx\,dy\right)^\frac{1}{p}\,.
\]
Let also $1\le q<p^\ast_s$
(here $p^\ast_s=Np/(N-sp) $ if $sp<N$ and $p^\ast_s=\infty$ otherwise.)
The continuity of the embedding of $\mathcal{D}_0^{s,p}(\Omega)$ into $L^q(\Omega)$ is equivalent
to condition $\lambda_{p,q}^s(\Omega)>0$, where 
\begin{equation}
\label{1.5}
\lambda_{p,q}^s(\Omega) = \inf_{u\in C^\infty_0(\Omega)} \Big\{ \iint_{\mathbb R^{2N}} \frac{|u(x)-u(y)|^p}{|x-y|^{N+sp}}\,dx\,dy\colon
\int_\Omega|u|^q\,dx=1 \Big\}\,.
\end{equation}
One aim of this short note is to relate this condition to the {\em compactness} of the embedding.
To do so, following~\cite{BR} we combine variational techniques and comparison principles and
in Section~\ref{s:2} we give, for {\em every} open set $\Omega$, a suitable weak definition of the
unique solution $w_{s,p,\Omega}$ of the problem
\begin{equation}
\label{1.7}
\begin{cases}
	 (-\Delta_p)^{s}\,w=1\,, & \qquad \text{in $\Omega$,}\\ w=0\,, & \qquad \text{in $\mathbb R^n\setminus\Omega$.}
\end{cases}
\end{equation}
The {\em $(s,p)$-laplacian} $ (-\Delta_p)^{s}$ is the integro-differential
operator defined (up to renormalisations) by
\begin{equation}
\label{1.6}
(-\Delta_p)^su(x):=- 2\lim_{\varepsilon\to0^+}\int_{|x-y|>\varepsilon} \frac{|u(x)-u(y)|^{p-2}(u(x)-u(y))}{|x-y|^{N+sp}}\,dy\,,
\end{equation}
for all smooth functions $u$.
The function $w_{s,p,\Omega}$ is to be called the $(s,p)$-torsion function on $\Omega$,
since (formally) for $s=1$ the solution of \eqref{1.7} is
the $p$-torsion function on $\Omega$. 
\medskip

First, we have the following result.

\begin{teo}\label{main}
Let $\Omega\subset\R^N$ be open, $1<p<\infty$, $1\le q<\infty$, and $0<s<1$. Then the following holds:
\begin{itemize}
\item If $1\le q<p$, then
	\begin{equation}\label{subhom}
		\lambda_{p,q}^s(\Omega)>0 \Longleftrightarrow w_{s,p,\Omega} \in L^{\frac{p-1}{p-q}q}(\Omega)\,.
	\end{equation}
\item If $p\le q<p^\ast_s$ then
		\begin{equation}\label{superhom}
		\lambda_{p,q}^s(\Omega) >0\Longleftrightarrow w_{s,p,\Omega}\in L^\infty(\Omega)\,.
	\end{equation}
\end{itemize}
\end{teo}

We also present a consequence of Theorem~\ref{main}, concerning super-homogeneous embeddings.
We refer to~\cite{Mb} for a different proof in Sobolev spaces for $s=1$.
\begin{cor}\label{mainscemo}
Let $1<p<\infty$, let $0<s<1$, and let $\Omega$ be an open set. Then
 $\mathcal{D}_0^{s,p}(\Omega)\hookrightarrow L^p(\Omega)$
if and only if
$\mathcal{D}_0^{s,p}(\Omega)\hookrightarrow L^q(\Omega)$ for some  (hence for all) $q$ with $p<q<p^\ast_s$.
\end{cor}

Next, we provide a criterion for the compactness of sub-homogeneous Sobolev embeddings. 

\begin{teo}\label{main2}
Let  $1\le q<p$, let $0<s<1$, and let $\Omega$ be an open set in $\R^N$.
Then the compactness of the embedding
$\mathcal{D}_0^{s,p}(\Omega)\hookrightarrow L^q(\Omega)$ is equivalent both to to the finiteness of
$\|w_{s,p,\Omega}\|_{\frac{p-1}{p-q}q}$
and to the positivity of
$\lambda_{p,q}^s(\Omega)$.
\end{teo}

The proofs of these results are presented in Section~\ref{s:5} and they rely a new Hardy-type inequality, involving
the $(s,p)$-torsion function, proved in Section~\ref{s:4}.

\section{Preliminaries}\label{s:2}

Throughout this note, 
for every open set $\Omega$ in the Euclidean $N$-space $\mathbb R^N$
we will denote by $C^\infty_0(\Omega)$ the set of all $C^\infty$ smooth functions with compact support in $\Omega$.
Given $s\in(0,1)$ and $p\in(1,\infty)$, we define
$\mathcal{D}_0^{s,p}(\Omega)$ as the completion of $C^\infty_0(\Omega)$ with respect to the norm 
\begin{equation}\label{1.1}
	[u]_{s,p}=\left\{\iint_{\mathbb R^{2N}} \frac{|u(x)-u(y)|^p}{|x-y|^{N+sp}}\,dx\,dy\right\}^\frac{1}{p}\,,
	\qquad  u\in C^\infty_0(\Omega)\,.
\end{equation} 

A list of properties of $\mathcal{D}_0^{s,p}(\Omega)$ is given e.g. in~\cite{BLP}, see in particular Section 2 and Appendix B therein.
We summarise here a couple of facts we shall need in the sequel.

If $\Omega$ is bounded in one direction, in view of~\cite[Lemma 5.2]{BC}
we get $\mathcal{D}_0^{s,p}(\Omega)$ by completion also starting from the norm
\begin{equation}\label{1.2}
\|u\|_{L^p(\Omega)}+ [u]_{s,p}\,.
\end{equation}
Instead, for a general open set the two procedures are not equivalent and
adding the $L^p$ norm results in a smaller space unless
$\Omega$ supports a fractional Poincar\'e inequality, i.e., if there exists $\lambda>0$ with
\begin{equation}
\label{1.4}
	\lambda\int_\Omega|u|^p\,dx \le  \iint_{\mathbb R^{2N}} \frac{|u(x)-u(y)|^p}{|x-y|^{N+sp}}\,dx\,dy\,,
	\qquad \text{for all $u\in C^\infty_0(\Omega)$.}
\end{equation}
In fact, in general $\mathcal{D}_0^{s,p}(\Omega)$ is not a space of distributions, either
(for some examples, we refer the interested reader, e.g., to~\cite{DL,HL})

Incidentally, if in addition $sp\neq1$ and $\Omega$ has a Lipschitz regular boundary then
$\mathcal{D}_0^{s,p}(\Omega)$ coincides with the subspace $W^{s,p}_0(\Omega)$ of the Sobolev-Slobodecki\u{\i}
space $W^{s,p}(\Omega)$, given by the closure in  $W^{s,p}(\Omega)$ of $C^\infty_0(\Omega)$ with respect to a norm different from \eqref{1.2}, more precisely the following one:
\begin{equation}
\label{1.3}
\Big(\int_\Omega|u|^p\,dx\Big)^\frac{1}{p}+\Big(\int_{\Omega}\int_{\Omega} \frac{|u(x)-u(y)|^p}{|x-y|^{N+sp}}\,dx\,dy\Big)^\frac{1}{p}\,.
\end{equation}
On the contrary, the existence of functions $u\in W^{s,p}_0(\Omega)$ for which the integral
\begin{equation}
\int_\Omega\int_{\mathbb R^N\setminus \Omega} \frac{|u(x)|^p}{|x-y|^{N+sp}}\,dx\,dy
\end{equation}
is infinite cannot be ruled out except if the boundary of $\Omega$ is smooth, hence in general
 $\mathcal{D}_0^{s,p}(\Omega)$ is a narrower space than $W^{s,p}_0(\Omega)$, even if $\Omega$ is bounded.

We set
\[
	p_s^\ast = 
	\begin{cases}
	\frac{Np}{N-sp}\,, & \text{if $sp<N$,}\\ 
	\infty\,, & \text{if $sp\ge N$.}
	\end{cases}
\]
In cases when $sp{<}N$, 
$\mathcal{D}_0^{s,p}(\Omega)$ is indeed a function space, thanks to the embedding
of $\mathcal{D}_0^{s,p}(\Omega)$ into $L^{p^\ast_s}(\Omega)$.
In these cases, 
the best constant in the Sobolev embedding, i.e.,
\begin{equation*}
\inf \Big\{ [u]_{s,p}^p \colon \|u\|_{L^{p^\ast_s}(\Omega)}=1\Big\}\,,
\end{equation*}
is independent of $\Omega$ and here will be denoted by $\mathcal{S}(N,s,p)$.
We refer, e.g., to~\cite{BMS,MS} for a more detailed account about
this constant and the extremals, viz. the functions $u$ for which inequality
\begin{equation}
\label{fracsob}
\mathcal{S}(N,s,p)\left(\int_{\R^N}|u|^{\frac{Np}{N-sp}}\,dx\right)^\frac{N-sp}{N}\le\iint_{\R^{2N}} \frac{|u(x)-u(y)|^p}{|x-y|^{N+sp}}\,dx\,dy
\end{equation}
holds as an equality.

The following Lemma contains a well known fact about functions in the Campanato
space
 $\mathcal L^{p,\lambda}$ with $\lambda>N$. We include a proof for the convenience of the reader.
Given $u \in C^\infty_0(\R^N)$, we denote by $u_{x,r}=\intmed_{B(x,r)} u\,dy$ the average
of $u$ on the ball $B(x,r)$ of radius $r$ about $x\in\mathbb R^N$. 

\begin{lm}\label{lmCA}
Let $C>0$ and let $u \in C^\infty_0(\R^N) $ with
\begin{equation}
\label{Ca1}
 \intmed_{B(x,r)} |u-u_{x,r}|^p\,dy \le C r^{sp-N}\,,\quad\quad \text{for all $x\in\R^N$ and for all $r>0$.}
\end{equation}
Then
\begin{equation}
\label{Ca2}
	|u(x)-u_{x,r}|\le c(N,s,p)\cdot C r^{s-\frac{N}{p}}\,,\quad
	\text{for all $x\in\mathbb R^N$ and for all $r>0$.}
\end{equation} 
\end{lm}
\begin{proof}
It is enough to show that
\begin{equation}
\label{estimC}
|u_{x,2^{-k}r}-u_{x,2^{-(k+h)}r}| \le  
c(N,s,p)\cdot C\, \frac{1-2^{-h(\frac{sp-N}{p})}}{2^{k\left(s-\frac{N}{p}\right)}} r^{s-\frac{N}{p}}\,,
\end{equation}
for all $x\in \R^N$, for all $r>0$, and for all $k,h\in\mathbb N$. Indeed, \eqref{estimC} implies that $(u_{x,2^{-h}r})_{h\in\mathbb N}$
is a Cauchy sequence. Then, taking $k=0$ and passing to the limit as $h\to\infty$ 
in \eqref{estimC} we obtain \eqref{Ca2}.

To prove \eqref{estimC}, we fix $k $ and we denote by $u_h $
the average of $u$ on the ball of radius $2^{-(h+k)}r$ centred at $x$. 
Because of triangle inequality, \eqref{estimC} holds if for every $h$ we have
\begin{equation}
\label{estimCR}
	|u_{j-1}-u_j|\le \omega_N^{s-\frac{1}{p}}2^{1+\frac{N}{p}}   2^{-(j-1)\left(s-\frac{N}{p}\right)}\cdot C R^{s-\frac{N}{p}}
	\,,\qquad \text{for all $j=1,\ldots,h$,}
\end{equation}
with $R=2^{-k}r$. To see that \eqref{estimCR} holds, we observe that
for every $y\in B(x,2^{-j}R)$ we have
\[
	2^{1-p} |u_{j-1}-u_j|^p\le |u_{j-1}-u(y)|^p+|u(y)-u_j|^p\,.
\]
Then an integration over $B(x,2^{-j}R)$, together with straightforward estimates, by \eqref{Ca1} gives
\[
	|u_{j-1}-u_1|^p\le \omega_N^{sp-1} 2^{(1-s)p}  (2^{-j})^{sp-N} \cdot C^p R^{sp-N}\,,
\]	
and taking the $p$-th root we obtain  \eqref{estimCR}.
\end{proof}
\begin{rmq}
The fact that $u\in C^{0,\alpha}(\R^N)$, with $\alpha=s-\frac{N}{p}$,
can be deduced with ease from \eqref{Ca2}.
\end{rmq}

The following Gagliardo-Nirenberg interpolation inequalities will be used a number of times in the rest of the paper.
For every $\gamma>1$ and for every function $u$, we abbreviate $\|u\|_{L^\gamma(\R^N)}$ to $\|u\|_\gamma$.
\begin{lm}\label{lm:GN}
Let $1\le q\le p<\infty$ and let $0<s<1$. Then the following holds:
\begin{itemize}
\item
\case{$sp\neq N$}{
	for every $r>0$ with $q<r\le p^\ast_s$ and for every $u\in C^\infty_0(\R^N)$ we have
	\begin{equation}
	\label{GN1}
		\left(\int_{\R^N}|u|^r\,dx\right)^{\frac{1}{r}}\le
		C_1\left( \int_{\R^N} |u|^q\,dx\right)^\frac{1-\vartheta}{q}
		\left( 
		\iint_{\R^{2N}} \frac{|u(x)-u(y)|^p}{|x-y|^{N+sp}}\,dx\,dy
		\right)^{\frac{\vartheta}{p}}
	\end{equation}
	with $\vartheta = \left(1-\frac{q}{r}\right)\left(1+\frac{sp-N}{Np}q\right)^{-1}$,
	for a suitable $C_1=C_1(N,p,q,r,s)>0$;
}
\item \case{$sp=N$}{
	for every $r\ge N/s$
	 and for every $u\in C^\infty_0(\R^N)$ we have
	\begin{equation}
	\label{GN2}
		\left(\int_{\R^N}|u|^r\,dx\right)^{\frac{1}{r}}
		\le C_2\left(\int_{\R^N} |u|^q\,dx\right)^{\frac{1}{r}}
		\!\!\left(
		\iint_{\R^{2N}} \frac{|u(x)-u(y)|^\frac{N}{s}}{|x-y|^{2N}}\,dx\,dy
		\right)^{\frac{s}{N}\left(1-\frac{q}{r}\right)}\,
	\end{equation}
	for a suitable $C_2= C_2(N,r,s) >0$.
}
\end{itemize}
\end{lm}

\begin{rmq}
Since $q\le p$, an inequality of the form
\[
\left( \int_{\R^N}|u|^r\,dx\right)^\frac{1}{r} \lesssim \left(\int_{\R^N}|u|^q\right)^\frac{\alpha}{q} 
\left(
\iint_{\R^{2N}} \frac{|u(x)-u(y)|^p}{|x-y|^{N+sp}}\,dx\,dy
\right)^\frac{\beta}{p}\,,\qquad \text{for all $u\in C^\infty_0(\R^N)$,}
\]
can hold for a unique (ordered) pair $(\alpha,\beta)$ of exponents. This fact is easily checked
by the invariance of the inequality under vertical and horizontal scalings.
\end{rmq}

\begin{proof}[Proof of Lemma~\ref{lm:GN}]
In the case $sp<N$, \eqref{GN1} is a direct consequence of the fractional Sobolev inequality
\eqref{fracsob}
combined with the standard interpolation inequality
\begin{equation*}
	\|u\|_{r}\le
	\|u\|_{q}^{1-\vartheta}\|u\|_{p^\ast_s}^{\vartheta}
\end{equation*}
with $\vartheta ={ \left(1-\frac{q}{r}\right)}{\left(1-\frac{sp-N}{Np}q\right)^{-1}}$, and 
in this case
\begin{equation}
\label{SobGN}
	C_1
	=\left[ \mathcal{S}(N,s,p)\right]^{-\frac{\vartheta}{p}}\,.
\end{equation}

To prove \eqref{GN1} in the case $sp>N$, we fix $u\in C^\infty_0(\R^N)$, $ x\in\R^N$, $r>0$, and we observe that 
\begin{equation*}
	\int_{B(x,r) }|u-u_{x,r}|^p\,dy \le \frac{1}{\omega_N r^N} \int_{B(x,r)}\int_{B(x,r)} |u(y)-u(z)|^p\,dy\,dz\,,
\end{equation*}
by Jensen inequality.
Since $|y-z|<2r$ for all $y,z\in B(x,r)$, we deduce 
\begin{equation*}
	\int_{B(x,r) }|u-u_{x,r}|^p\,dy \le \frac{2^{N+sp}}{\omega_N} r^{sp} \iint_{\R^{2N}} \frac{|u(y)-u(z)|^p}{
	|z-y|^{N+sp}}\,dy\,dz\,.
\end{equation*}
By Lemma~\ref{lmCA}, this implies that 
\begin{equation*}
|u(x)-u_{x,r}|\le c \left( \iint_{\R^{2N}} \frac{|u(y)-u(z)|^p}{
	|z-y|^{N+sp}}\,dy\,dz\right)^\frac{1}{p} r^{s-\frac{N}{p}}\,,
\end{equation*}
for a suitable constant $c=c(N,s,p) >0$.
By H\"older inequality we have
\begin{equation*}
|u_{x,r}| \le \left( \intmed_{B(x,r)} |u(y)|^q\,dy\right)^\frac{1}{q} \,.
\end{equation*}
The last two inequalities hold for all $x\in\R^N$ and for all $r>0$, in particular with $r=1$. Therefore 
\begin{equation*}
\|u\|_{L^\infty(\R^N)} \le \left( \int_{\R^N} |u|^q\,dx\right)^\frac{1}{q} + c(N,s,p)\left( \iint_{\R^{2N}}\frac{|u(x)-u(y)|^p}{|x-y|^{N+sp}}\,dx\,dy\right)^{\frac{1}{p}}\,.
\end{equation*}
By a standard homogeneity argument, based on the invariance under horizontal scalings, the latter can be rephrased in the
following multiplicative form
\begin{equation*}
	\|u\|_{L^\infty(\R^N)} \le C(N,s,p) \left(\int_{\R^N}|u|^q\,dx\right)^{\frac{sp-N}{Np+(sp-N)q}}
	\left( \iint_{\R^{2N}} \frac{|u(x)-u(y)|^p}{|x-y|^{N+sp}}\,dx\,dy\right)^\frac{N}{Np+(sp-N)q}\,.
\end{equation*}
Then \eqref{GN1} follows by the obvious estimate
\(
	\|u\|_r\le \|u\|_{\infty}^{1-\frac{q}{r}} \|u\|_q^\frac{q}{r}\,.
\)

Eventually, to end the proof we assume that $1\le q \le p=N/s \le r$ and we prove \eqref{GN2}. 
Let $\sigma = \frac{3s}{4}$
and set $\theta {=} (1-\frac{p}{r}) \frac{N}{\sigma p}$.
Since $\sigma p<N$, applying \eqref{GN1} with $q=p$ and $s $ replaced by $\sigma$ we get
\begin{equation}\label{gnpN1}
\left(\int_{\R^N}|u|^r\right)^{\frac{1}{r}}\le
C(N,r,s) \left(\int_{\R^N}|u|^{\frac{N}{s}}\right)^{\frac{s(1-\theta)}{N}}
\left( \iint_{\R^{2N}}\frac{|u(x)-u(y)|^{\frac{N}{s}}}{|x-y|^{N(1+\frac{\sigma}{s})}}\right)^\frac{s\theta}{N}
\end{equation}
for all $u\in C^\infty_0(\R^N)$. 
We observe that the inequality
\begin{equation}\label{gnpN2}
\left( \iint_{\R^{2N}}\frac{|u(x)-u(y)|^{\frac{N}{s}}}{|x-y|^{N(1+\frac{\sigma}{s})}}\right)^\frac{s\theta}{N}
\le C(s)  \left(\int_{\R^N}|u|^\frac{N}{s}\right)^{\frac{s-\sigma}{\sigma \frac{N}{s}}}
\left( \iint_{\R^{2N}}\frac{|u(x)-u(y)|^{\frac{N}{s}}}{|x-y|^{2N}}\right)^\frac{\sigma}{N}
\end{equation}
holds for all $u\in C^\infty_0(\R^N)$, too. Indeed, since $\sigma<s$ we have
\[
	\int_{\R^N} \int_{|y-x|<1} \frac{|u(x)-u(y)|^\frac{N}{s}}{|x-y|^{N(1+\frac{\sigma}{s})}}\,dx\,dy \le
	\int_{\R^N} \int_{|y-x|<1} \frac{|u(x)-u(y)|^\frac{N}{s}}{|x-y|^{2N}}\,dx\,dy \,.
\]
In addition, we also have that
\[
	\int_{\R^N}\! \int_{|x-y|\ge1} \frac{|u(x)-u(y)|^\frac{N}{s}}{|x-y|^{N(1+\frac{\sigma}{s})}}\,dx\,dy \le
	2^{\frac{N}{s}}\!\int_{\R^N} |u(x)|^\frac{N}{s}
\!	\int_{|y-x|\ge1}\frac{dy}{|x-y|^{N(1+\frac{\sigma}{s})}}
	dx
	\le \frac{2^{\frac{N}{s}+1}}{N}\! \int_{\R^N} |u|^\frac{N}{s}\,dx\,,
\]
where in the last passage we used that $\sigma>s/2$. Then, \eqref{gnpN2} follows
by a direct homogeneity argument.
Combining \eqref{gnpN1} and \eqref{gnpN2} with standard interpolation in Lebesgue spaces we obtain
\begin{equation}\label{GN1.9}
\|u\|_{L^r(\R^N)} \le C_2(N,r,s) \|u\|_{L^q(\R^N)}^{(1-\lambda)(1-\frac{\theta \sigma}{s})} \|u\|_{L^r(\R^N)}^{\lambda(1-\frac{\theta\sigma}{s})} \left( \iint_{\R^{2N}}\frac{|u(x)-u(y)|^\frac{N}{s}}{|x-y|^{2N}}\,dx\,dy \right)^{\frac{\theta\sigma}{s}}
\end{equation}
for all $u\in C^\infty_0(\R^N)$, with $\lambda\in(0,1)$ being such that $\frac{s}{N} = \frac{1-\lambda}{q} + \frac{\lambda}{r}$.
We observe that by definition we have
\(
	\frac{\theta\sigma}{s} = 1-\frac{N}{rs}
\). 
Then \eqref{GN2} follows
dividing out a term in \eqref{GN1.9}.
\end{proof}

\begin{rmq}
The proof above works with no difference if $\sigma=\frac{3s}{4}$ is replaced by any other $\sigma\in(\frac{s}{2},s)$.
In this case, the constant appearing
in \eqref{GN2} will change, going to depend on the choice of $\sigma$
through the one appearing in \eqref{gnpN1}. Note that in view of \eqref{SobGN} the latter blows up as $\sigma\to s^{-}$.
\end{rmq}

\section{The Fractional Torsion Function}\label{s:2}

\subsection{Compact case} Throughout the present subsection,
we shall assume that the embedding of $\mathcal{D}_0^{s,p}(\Omega) $ into $L^1(\Omega)$ is compact, and we  list some properties of the fractional torsion function under this assumption.

\begin{defi}
Let $\Omega$ be such that  the embedding of $\mathcal{D}_0^{s,p}(\Omega) $ into $L^1(\Omega)$ is compact.
Then we call the $(s,p)$-torsion function on $\Omega$, denoted by $w_{s,p,\Omega}$, the unique 
solution of the minimum problem
\begin{equation}
\label{3.1}
\min_{u\in \mathcal{D}_0^{s,p}(\Omega)}\left\{\frac{1}{p}\iint_{\mathbb R^{2N}} \frac{|u(x)-u(y)|^p}{|x-y|^{N+sp}}\,dx\,dy-\int_\Omega u\,dx\right\}\,.\qquad
\end{equation}
\end{defi}
By a standard homogeneity argument, the minimum value in \eqref{3.1} equals
$-\frac{p-1}{p} ( T_{s,p}(\Omega))^\frac{1}{p-1} $, where the {\em $(s,p)$-torsional rigidity} is defined by
\begin{equation}
\label{Trigidity}
	 T_{s,p}(\Omega):=\max\big\{\|u\|_{L^1(\Omega)}^p\colon u\in \mathcal{D}_0^{s,p}(\Omega)\,, \ [u]_{s,p}^p=1\big\} \,.
\end{equation}

We point out that no Lavrentiev's phaenomenon occurs between $\mathcal{C}^\infty_0(\Omega)$ and
$\mathcal{D}_0^{s,p}(\Omega)$ in \eqref{3.1}. 
More precisely, we get the same value in \eqref{3.1} if
instead of minimising over $\mathcal{D}_0^{s,p}(\Omega)$
we take the infimum over $C^\infty_0(\Omega)$.
Indeed, it is clear that the latter is a quantity greater than or equal to \eqref{3.1},
due to the inclusion
$C^\infty_0(\Omega)\subset \mathcal{D}_0^{s,p}(\Omega)$, and
the reverse inequality also holds
by the definition of  $\mathcal{D}^{s,p}_0(\Omega)$ and by the compactness of its
embedding in $L^1(\Omega)$.

Since
the $(s,p)$-torsion function on $\Omega$ is obtained by minimizing a convex energy on $\mathcal{D}_0^{s,p}(\Omega)$,
it is the unique solution of
\begin{equation}
\label{3.1eq}
\iint_{\mathbb R^{2N}} \frac{|w(x)-w(y)|^{p-2}(w(x)-w(y))}{|x-y|^{N+sp}}(\varphi(x)-\varphi(y))\,dx\,dy
=\int_\Omega\varphi\,dx\,,\quad \text{for all $\varphi\in C^\infty_0(\Omega)$.}
\end{equation}
Simbolically, the Euler-Lagrange equation \eqref{3.1eq} can be written in the form \eqref{1.7}.

\begin{prop}\label{locale}
If $\mathcal{D}_0^{s,p}(\Omega)\hookrightarrow L^1(\Omega)$
is compact, then $w_{s,p,\Omega}\in L^\infty(\Omega)$. Moreover, if $sp<N$,
\begin{equation}
\label{2.1.1}
\|w_{s,p,\Omega}\|_{L^\infty(\Omega)} \le  \frac{N+sp'}{sp'}\mathcal{S}(N,s,p)^\frac{N}{Np+sp-N}  \left(\int_\Omega w_{s,p,\Omega}\,dx\right)^\frac{sp'}{N+sp'}\,.
\end{equation}
\end{prop}
\begin{proof}
Let us abbreviate $w_{s,p,\Omega}$ to $w$. If $sp>N$, \eqref{2.1.1} is a direct consequence of the Gagliardo-Nirenberg type inequality \eqref{GN1}, with $q=1$ and $r=\infty$, hence we may assume that $sp\le N$.

We first prove \eqref{2.1.1} in the case when $sp<N$. To do so, we fix
$k>0$ and we note that
the function defined by truncation setting $\varphi_k(x)=\max\{ w(x)-k,0\}$,
is an admissible test function for \eqref{3.1eq}.
We let $A_k=\{x\in\mathbb R^N\colon w(x)>k\}$ and we observe that 
the set $\mathcal{W}_k=A_k\times (\mathbb R^N\setminus A_k)$ 
is contained in $\{(x,y)\in \mathbb R^N\times\mathbb R^N\colon w(x)-w(y)\ge w(x)-k\ge0\}$.  Therefore
\begin{equation}
\label{2.1.2}
\begin{split}
\iint_{\mathcal{W}_k}
	\frac{|w(x)-w(y)|^{p-2}}{|x-y|^{N+sp}} (w(x)-w(y))(w(x)-k)\,dy\,dx
	\ge
	\iint_{\mathcal{W}_k}\frac{|\varphi_k(x)-\varphi_k(y)|^p}{|x-y|^{N+sp}}\,dx\,dy\,.
\end{split}
\end{equation}
Moreover
we have
\begin{subequations}
\label{2.1.3}
\begin{equation}
	\iint_{(\mathbb R^N\setminus A_k)\times (\mathbb R^N\setminus A_k)} \frac{|\varphi_k(x)-\varphi_k(y)|^p}{|x-y|^{N+sp}}\,dx\,dy=0\,,
\end{equation}
and
\begin{equation}
\begin{split}
\iint_{A_k\times A_k} 
\frac{|w(x)-w(y)|^{p-2}(w(x)-w(y))}{|x-y|^{N+sp}} & (\varphi_k(x)-\varphi_k(y))
 = \iint_{A_k\times A_k}
\frac{|\varphi_k(x)-\varphi_k(y)|^p}{|x-y|^{N+sp}}\,.
\end{split}
\end{equation} 
\end{subequations}
By the symmetry of
the left hand-side in \eqref{3.1eq} with respect to $(x,y)\mapsto(y,x)$, when  plug in $\varphi_k$ into \eqref{3.1eq}, combining \eqref{2.1.2} with the identities
\eqref{2.1.3} we arrive at
\begin{equation}
\label{2.1.5}
\int_{\mathbb R^N}\int_{\mathbb R^N} \frac{|\varphi_k(x)-\varphi_k(y)|^p}{|x-y|^{N+sp}}\,dx\,dy \le \int_{A_k}(w(x)-k)^p\,dx\,.
\end{equation}
On the other hand, by \eqref{fracsob}, we have that
\begin{equation}
\label{2.1.6}
	\int_{\mathbb R^n}\int_{\mathbb R^n} \frac{|\varphi_k(x)-\varphi_k(y)|^p}{|x-y|^{N+sp}}\,dx\,dy 
	\ge \mathcal{S}(N,s,p) |A_k|^{1-p-\frac{sp}{N}} \left(\int_{A_k}(w(x)-k)\,dx\right)^p\,.
\end{equation}
By Fubini's theorem, using the estimates \eqref{2.1.5} and \eqref{2.1.6} and dividing out, we obtain
\begin{equation}
\label{2.1.7}
\Big(\int_k^\infty|A_t|\,dt\Big)^{p-1}\le \mathcal{S}(N,s,p)^{-1} |A_k|^{-1+p+\frac{sp}{N}}\,.
\end{equation}
Since $w\in L^1(\Omega)$, $k\mapsto |A_k|$ is a non-increasing function converging to $0$ as $k\to\infty$.
Thus by \eqref{2.1.7} the function $\varepsilon(k)=\int_k^\infty|A_t|\,dt$ satisfies the differential inequality
\begin{equation}
\label{2.1.8}
\varepsilon(k)^\frac{N}{N+sp'}\le C(N,s,p)\big( -\varepsilon'(k)\big)
\end{equation}
with $C= \mathcal{S}(N,s,p)^\frac{-N}{N(p-1)+sp}$.
This gives that $w\in L^\infty(\Omega)$. Indeed, given $k_0>0$ and $k>k_0$ by integration we infer from \eqref{2.1.8} that
\begin{equation}
\label{2.1.9}
	k-k_0\le C \frac{N+sp'}{sp'} \big(\varepsilon(k_0)^\frac{sp'}{N+sp'}-\varepsilon(k)^\frac{sp'}{N+sp'}\big)\,.
\end{equation}
To get the quantitative bound \eqref{2.1.1},
we observe that \eqref{2.1.9} implies $\varepsilon(k)=0$ whenever
\begin{equation}\label{2.1.10}
k\ge k_0+ C \frac{N+sp'}{sp'} \Big( \int_{A_{k_0}}(w-k_0)\,dx\Big)^\frac{sp'}{N+sp'}\,.
\end{equation}
Clearly this implies that $|A_k|=0$ for $k$ satisfying \eqref{2.1.10}. Since we may take any $k_0>0$ in the lower bound \eqref{2.1.10},
this and the definition of $C$ give \eqref{2.1.1}.

To end the proof, the only case left to consider is that when $sp=N$. In this case, applying \eqref{GN2}
with exponents $q=1$ and $r=\frac{tN}{s}$, with $t>1$, and arguing as in the previous case we obtain
\begin{equation}
\label{borderline-decay}
	\varepsilon(k)^{\beta(t)}\le C(N,t,s)(-\varepsilon'(k))\,,
	\quad
	\text{with $\beta(t)= \frac{tp(p-1)-(t-1)((t+1)p-1)}{(tp-1)(p-1)}$.}
\end{equation}
Eventually, we choose $t>1$ so that $\beta(t) = 1-\frac{s}{N}$ and arguing as before we get \eqref{2.1.1}.
\end{proof}
We refer to~\cite{LL} for the following weak comparison principle. Similar results
have been proved in slightly different settings, see~\cite{BMS,IMS}
\begin{prop}\label{comparLL}
Let $w_i=w_{s,p,\Omega_i}$ where $\Omega_i$ is a bounded open set. If
$\Omega_1\subset\Omega_2$ then $w_1\le w_2$.
\end{prop}
\begin{proof}
Setting
\[
	J(w_i,\varphi) = \int_{\mathbb R^N}\int_{\mathbb R^N} \frac{|w_i(x)-w_i(y)|^{p-2}(w_i(x)-w_i(y))(\varphi(x)-\varphi(y))}{|x-y|^{N+sp}}\,dx\,dy
\]
clearly we have 
\[
	J(w_2,\varphi)-J(w_1,\varphi) = \int_{\Omega_2\setminus \Omega_1}\varphi\,dx\ge0\,,
\]
for all $\varphi\in C^\infty_0(\Omega_2)$ with $\varphi\ge0$. The conclusion then follows arguing as in~\cite[Lemma 9]{LL}.
\end{proof}

\subsection{The general case} In view of Proposition~\ref{comparLL}, we can define the fractional torsion function 
on arbitrary open sets $\Omega\subset\R^N$ as follows.
\begin{defi}\label{T}
Given an open set $\Omega\subset\mathbb R^N$ the
{\em $(s,p)$-torsion function} of $\Omega$ is defined by 
\begin{equation}
\label{T}
	w_{s,p,\Omega}(x) = \lim_{r\to\infty} w_{r}(x)\,, \qquad \text{for every $x\in\Omega$,}
\end{equation}
where we set
\begin{equation}\label{TT}
w_{r}(x) = \begin{cases}
w_{B_r(0)\cap\Omega}(x)\,, & \qquad \text{if $x\in B_r(0)\cap\Omega$,}\\
0\,, &\qquad \text{otherwise,}
\end{cases}
\end{equation}
for all $r>r_0=\inf\{ \rho>0\colon |B_\rho(0)\cap \Omega|>0\}$.
\end{defi}
We shall often identify $w$ with its extension to the whole space $\R^N$ with $w\equiv 0$ in $\R^N\setminus\Omega$.

\begin{rmq}
Note that the torsion function is well defined.
First of all the limit in \eqref{T} makes sense
by Proposition~\ref{comparLL}.
Moreover, for every open set $\Omega$ for which the embedding $\mathcal{D}_0^{s,p}(\Omega)\hookrightarrow L^1(\Omega)$
is compact, the function
 $w_r$  converges, as $r\to\infty$, to the unique solution  of \eqref{3.1}.
Indeed, using $w_r$ first as a test function in its equation (i.e., \eqref{3.1eq} with $\Omega\cap B_r(0)$
in place of $\Omega$)
and then as a competitor in \eqref{Trigidity} (see~\cite[Lemma 2.4]{BR} where a similar task is carried out in detail) we get
\[
[w_r]_{s,p}^p  = \|w_r\|_{L^1(\Omega)} \le [w_r]_{s,p} T_{s,p}(\Omega)^\frac{1}{p}\,,
\]
and we conclude by the reflexivity of $\mathcal{D}_0^{s,p}(\Omega)$ and the compactness
of its embedding in $L^1(\Omega)$.
\end{rmq}

\begin{rmq}
We point out that $w_{s,p,\Omega}>0$ in $\Omega$.
To see this we may assume with no restriction $\Omega$ to be bounded,
since  \eqref{T} is a pointwise monotone limit.
Then the embedding $\mathcal{D}_0^{s,p}(\Omega)\hookrightarrow
L^1(\Omega)$ is compact, and $w_{s,p,\Omega}$ solves \eqref{3.1eq}.
Therefore, the conclusion in this case follows by the minimum principle (see, e.g.,~\cite[Appendix A]{BF}).
\end{rmq}

\section{Non-local Torsional Hardy inequalities}\label{s:4}

We begin this section with a fractional Hardy-type inequality involving the torsion function.

\begin{prop}\label{TH}
Let $1{<}p{<}\infty$, $0{<}s{<}1$. Let $\Omega\subset\R^N$ be an open set such that $\mathcal{D}_0^{s,p}(\Omega)\hookrightarrow L^1(\Omega)$ is compact. Then
\begin{equation*}
\int_\Omega \frac{|u|^p}{w_{s,p,\Omega}^{p-1}}\,dx \le \iint_{\R^{2N}} \frac{| u(x)-u(y)|^p}{|x-y|^{N+sp}}\,dx\,dy\,,
\quad \text{for all $u\in \mathcal{D}_0^{s,p}(\Omega)$.}
\end{equation*}
\end{prop}
\begin{proof}
We prove the inequality for any fixed $u\in\mathcal{D}_0^{s,p}(\Omega)$ with $u\ge0$, which is sufficient.
To do so, let $\varepsilon>0$, and
let $w=w_{s,p,\Omega}$. 
Since $f(t)=(t+\varepsilon)^{1-p}$, $t>0$,
is a Lipschitz function,  $\varphi=u^p(w+\varepsilon)^{1-p}$ is an admissible test function for equation \eqref{3.1eq} (see, e.g.,~\cite[Lemma 2.4]{BC}). Thus,
setting  $w_\varepsilon = w+\varepsilon$,
\[
\int_\Omega \frac{u^{p-1}}{(w+\varepsilon)^{p-1}}\,dx=\iint_{\R^{2N}} \frac{\big| w(x)-w(y)\big|^{p-2}\big(w(x)-w(y)\big)}{|x-y|^{N+sp}}\left( \frac{u(x)^p}{w_\varepsilon(x)^{p-1}}
-\frac{u(y)^p}{w_\varepsilon(y)^{p-1}}\right)\,dx\,dy \,.
\]
Hence, thanks to the following discrete Picone-type inequality (see, e.g.,~\cite[Proposition 4.2]{BF})
\begin{equation}
\label{piccone}
| a - b|^{p-2}(a-b) \left( \frac{c^p}{a^{p-1}}-\frac{d^p}{b^{p-1}}\right)\le |c-d|^p\,,
\qquad \text{for all $a,b>0$ and $c,d\ge0$,}
\end{equation}
we get the conclusion by Fatou's Lemma using the arbitrariness of $\varepsilon>0$.
\end{proof}
\begin{cor}Let $1{<}p{<}\infty$,  let $0{<}s{<}1$, and
let $\Omega\subset \R^N$ be any open set. Then
\begin{equation}\label{coroineq}
	\int_\Omega\frac{|u|^p}{w_{s,p,\Omega}^{p-1}}\,dx
	\le \iint_{\R^{2N}} \frac{|u(x)-u(y)|^p}{|x-y|^{N+sp}}\,dx\,dy\,,\qquad
	\text{for all $u\in C^\infty_0(\Omega)$}
\end{equation}
(with the convention that $\frac{c}{\infty}=0$ for all $c\in\R$.)
\end{cor}
\begin{proof}We fix $u\in C^\infty_0(\Omega)$ and let $R_0>0$ be such that, for every
 $R>R_0$, $u$ is supported in the ball $B_R$ of radius $R$ about the origin. Then, setting $\Omega_R=\Omega\cap B_R$,
 by Proposition~\ref{TH} we have
\[
	\int_{\Omega_R}\frac{|u|^p}{w_{s,p,\Omega_R}^{p-1}}\,dx
	\le \iint_{\R^{2N}} \frac{ |u(x)-u(y)|^p}{|x-y|^{N+sp}}\,dx\,dy\,,
\]
for all $R>R_0$. Thus, in view of
Definition~\ref{T}, the desired inequality follows by Fatou Lemma.
\end{proof}

We end this section with  a variation on the torsional Hardy inequality
discussed in Proposition~\ref{TH}, containing an additional term.

\begin{teo}\label{THR}
Let $\Omega\subset \R^N$ be such that the embedding $\mathcal{D}_0^{s,p}(\Omega)$ into $L^1(\Omega)$
is compact and let $w$ be the $(s,p)$-torsion function on $\Omega$. Then
there exists a constant $C>0$, only depending on $p$, with
\begin{equation*}
	\int_\Omega \frac{|u|^p}{w^{p-1}}\,dx + 2\int_{\R^{N}}\int_{\R^N}\left\vert \frac{w(x)-w(y)}{w(x)+w(y)}
	\right\vert^p \frac{dy}{|x-y|^{N+sp}}|u(x)|^p\,dx \le C \iint_{\R^{2N}} \frac{|u(x)-u(y)|^p}{|x-y|^{N+sp}}\,dx\,dy\,,
\end{equation*}
for all $u\in \mathcal{D}_0^{s,p}(\Omega)$.
\end{teo}
We skip the proof of Theorem~\ref{THR} because it is completely analogous to that of Proposition~\ref{TH}, except that 
instead of \eqref{piccone} one can exploit a Picone-type inequality
with a remainder term. More precisely, by~\cite[Lemma A.5]{BC} there exist positive constants 
$C_1,C_2$, only depending on $p$, with
\[
	|a-b|^{p-2}(a-b)\left( \frac{c^p}{a^{p-1}}-\frac{d^p}{b^{p-1}}\right)
	+C_1\left\vert\frac{a-b}{a+b}\right\vert^p(c^p+d^p)\le C_2 |c-d|^p\,,
\quad \text{for all $a,b>0$ and $c,d\ge0$.}
\]

\section{Proofs of the main results}\label{s:5}
For every open set $\Omega$ in $\R^N$,
for every $1<p<\infty$, $1\le q<p^\ast_s$, and $0<s<1$, we have
\begin{equation}
\label{poincare}
	\lambda_{p,q}^s(\Omega) \left(
		\int_\Omega|u|^q\,dx
	\right)^\frac{p}{q}
	\le
	\iint_{\R^{2N}} \frac{|u(x)-u(y)|^p}{|x-y|^{N+sp}}\,dx\,dy\,,\qquad
	\text{for all $u\in C^\infty_0(\Omega)$,}
\end{equation}
where
\begin{equation}
\label{lambda}
\lambda_{p,q}^s(\Omega) = \inf_{u\in C^\infty_0(\Omega)} \Big\{ \iint_{\mathbb R^{2N}} \frac{|u(x)-u(y)|^p}{|x-y|^{N+sp}}\,dx\,dy\colon
\int_\Omega|u|^q\,dx=1 \Big\}\,,
\end{equation}
and $\lambda_{p,q}^s(\Omega) $ is the best possible constant for this inequality to hold.

\begin{rmq}
\label{cauchypoincare}
The Poincar\'e-type inequality \eqref{poincare} implies that $\mathcal{D}_0^{s,p}(\Omega)$
is a function space, continuously included in $L^q(\Omega)$, whenever $\lambda_{p,q}^s(\Omega)>0$. Indeed, in this case, if $(u_n)_n\subset C^\infty_0(\Omega)$ is
 a Cauchy sequence in $\mathcal{D}_0^{s,p}(\Omega)$ then
 by \eqref{poincare} it is a Cauchy sequence in the Banach space $L^q(\Omega)$ as well.
\end{rmq}

We now prove Theorem~\ref{main}, relating the positivity of $\lambda_{p,q}^s(\Omega)$ to the summability of the $(s,p)$-torsion function;
this is the non-local counterpart of~\cite[Theorems 1.2, 1.3]{BR},
and the conclusion is obtained by adapting to the fractional framework the arguments used in~\cite{BR} in the local setting
(for $1\le q\le p$).
The proofs are different depending on whether
$1\le q<p$,  or $p\le q<p^\ast_s$.
\subsection{Proof of Theorem~\ref{main} (case $1\le q<p$)}
Let $w_R=w_{s,p,\Omega\cap B_R}$ and $\beta\ge1$. Since $t\mapsto t^{\beta}$ 
is locally Lipschitz continuous and $w_R\in L^\infty(\Omega)$ (by Proposition~\ref{locale}),
$\varphi=w_R^{\beta}$ is an admissible test function for equation \eqref{3.1eq}. Therefore
\begin{equation*}
\iint_{\R^{2N}} \frac{\big|w_R (x)-w_R (y)\big|^{p-2}\big(w_R (x)-w_R (y)\big)\big(w_R (x)^\beta-w_R (y)^\beta\big)}{|x-y|^{N+sp}}\,dx\,dy
=\int_\Omega w_R ^\beta\,dx\,.
\end{equation*}
Applying the elementary inequality (see~\cite[Lemma C.1]{BLP})
\[
	|a-b|^{p-2} (a-b)(a^\beta-b^\beta) \ge \beta \left[\frac{p}{p+\beta-1}\right]^p\left\vert 
		a^{\frac{\beta+p-1}{p}}-b^\frac{\beta+p-1}{p}
	\right\vert^p\,,
\]
with $a= w_R (x)$ and $b=w_R (y)$ and integrating, we deduce that
\begin{equation}\label{5.5}
 \beta \left[\frac{p}{p+\beta-1}\right]^p
 \iint_{\R^{2N}} \frac{\left\vert 
		w_R (x)^{\frac{\beta+p-1}{p}}-w_R (y)^\frac{\beta+p-1}{p}
	\right\vert^p}{|x-y|^{N+sp}}\,dx\,dy\le \int_\Omega w_R ^\beta\,dx\,.
\end{equation}
We observe that 
\begin{equation}\label{5.6}
 \iint_{\R^{2N}} \frac{\left\vert 
		w_R (x)^{\frac{\beta+p-1}{p}}-w_R (y)^\frac{\beta+p-1}{p}
	\right\vert^p}{|x-y|^{N+sp}}\,dx\,dy
	\ge \lambda_{p,q}^s(\Omega) \left(
		\int_{\Omega\cap B_R} w_R ^{\frac{\beta+p-1}{p}q}
	\right)^\frac{p}{q}\,,
\end{equation}
where we also used
the fact that $\lambda_{p,q}^s(\Omega)\le \lambda_{p,q}^s(\Omega\cap B_R)$,
in view  of the obvious
monotonicity of the quantity~\eqref{lambda}
with respect to set inclusion.

Combining \eqref{5.6} with \eqref{5.5} we get
\begin{equation*}
 \beta \left[\frac{p}{p+\beta-1}\right]^p \lambda_{p,q}^s(\Omega) \left(
		\int_{\Omega\cap B_R} w_R ^{\frac{\beta+p-1}{p}q}
	\right)^\frac{p}{q}
	\le 
	\int_\Omega w_R ^\beta\,dx\,.
\end{equation*}
Taking $\beta\ge1$ with $\beta = \frac{\beta+p-1}{p}q$, we obtain
\begin{equation}\label{usa}
\lambda_{p,q}^s(\Omega) \left(
		\int_{\Omega\cap B_R} w_R ^{\frac{p-1}{p-q}q}
	\right)^\frac{p-q}{q}
	\le \frac{1}{q}\frac{q-1}{p-1}\left( \frac{q-1}{p-q}\right)^{p-1}\,.
\end{equation}

Recall that $R>0$ was arbitrary. 
Hence, if $\lambda_{p,q}^s(\Omega)>0$, from \eqref{usa} we deduce
that
\begin{equation}\label{ub}
\|w_{s,p,\Omega}\|_{L^{\frac{p-1}{p-q}q}(\Omega)} \le
\left( \frac{1}{\lambda_{p,q}^s(\Omega)} \frac{q-1}{q(p-1)} \left( \frac{q-1}{p-q}\right)^{p-1}\right)^{\frac{p-1}{p-q}q}\,,
\end{equation}
by Definition~\ref{T} and Fatou's Lemma. This concludes the proof. \qed

\subsection{Proof of Theorem~\ref{main} (case $q\ge p$)} 
We first assume that 
$\lambda_{p,q}(\Omega)>0$ 
and we prove that 
$w:=w_{s,p,\Omega}$ belongs to $L^\infty(\Omega)$.
More precisely, we show that
\begin{equation}
\label{BVB}
\|w\|_{L^\infty(\Omega)}\le C\lambda_{p,q}^s(\Omega)^{\frac{1}{1-q}}\,.
\end{equation}
The argument is due to~\cite[Theorem 9]{VDBB}. 
Up to an approximation of $\Omega$ with an increasing sequence of smooth open sets,
while proving \eqref{BVB} we may assume without any restriction that $\Omega$ is itself smooth and bounded.
In particular, in  view of Proposition~\ref{locale}, we may assume that $\|w\|_{L^\infty(\Omega)}<+\infty$.
We shall also require that $w (0)=\|w\|_{L^\infty(\Omega)}$,
which again causes no loss of generality (we may assume this up to a translation).

Let $\zeta\in C^\infty_0(\R^N)$ be a cut-off function
from $B_{\frac{R}{2}}$ to $B_{R}$, with $|\nabla\zeta|\le 2R^{-1}$.
Since by our assumptions $w\in L^\infty(\Omega)$, 
the function $u=w\zeta$ is an admissible competitor for the variational
problem \eqref{lambda}, and we have
\begin{equation}
\label{lambdae}
\lambda_{p,q}^s(\Omega) \le \frac{\displaystyle\iint_{\R^{2N}} \frac{|w(x)\zeta(x)-w(y)\zeta(y)|^p}{|x-y|^{N+sp}}\,dx\,dy}{
\displaystyle \int_\Omega w(x)^p\zeta(x)^p\,dx}\,.
\end{equation}

We first estimate the numerator in \eqref{lambdae}. By Proposition~\ref{locale}, we can test equation
\eqref{3.1eq}
with $\varphi= w\zeta^p$ (see, e.g.,~\cite[Lemma 2.4]{BC}), so as to get
\begin{equation}\label{sottosol2}
\iint_{\mathbb R^{2N}} \frac{|w(x)-w(y)|^{p-2}(w(x)-w(y))}{|x-y|^{N+sp}}(w(x)\zeta(x)^p-w(y)\zeta(y)^p)\,dx\,dy
=\int_{B_R}w\zeta^p\,dx\,.
\end{equation}
The double integral appearing in \eqref{sottosol2} splits into its contributions in 
$\mathcal{C}^+=\{(x,y)\in\R^{2N}\colon |y|>|x|\}$ and $\mathcal{C}^-=\R^{2N}\setminus \mathcal{C}^+$.
Subtracting and adding terms, the two contributions read respectively as
\begin{subequations}\label{5.2.pm}
\begin{equation}
\label{5.2.1}
\iint_{\mathcal{C}^+}\!\! \frac{|w(x)-w(y)|^p}{|x-y|^{N+sp}}\zeta(x)^p
+\iint_{\mathcal{C}^+}\!\! \frac{|w(x)-w(y)|^{p-2}(w(x)-w(y))}{|x-y|^{N+sp}}w(y)(\zeta(x)^p-\zeta(y)^p)
\end{equation}
and
\begin{equation}
\label{5.2.2}
\iint_{\mathcal{C}^-}\!\! \frac{|w(x)-w(y)|^p}{|x-y|^{N+sp}}\zeta(y)^p
+\iint_{\mathcal{C}^-}\!\! \frac{|w(x)-w(y)|^{p-2}(w(x)-w(y))}{|x-y|^{N+sp}}w(x)(\zeta(y)^p-\zeta(x)^p)\,.
\end{equation}
\end{subequations}

Let $\mathcal{A}_1^+ = (B_R\times B_R)\cap \mathcal{C}^+$ and $\mathcal{A}_2^+=(B_R\times (\Omega\setminus B_R))\cap\mathcal{C}^+$.
We observe that $\zeta(x)\ge\zeta(y)$ in $\mathcal{C}^+$, whence it follows that
$\zeta(x)^p-\zeta(y)^p\le p\zeta(x)^{p-1}|x-y|$ for all $(x,y)\in\mathcal{A}_1^+$,
provided that we opted for a radially symmetric cut-off with a decreasing radial profile, and clearly we have
$\zeta(x)^p-\zeta(y)^p=\zeta(x)$ for all $(x,y)\in\mathcal{A}_2^+$. Therefore
\begin{equation*}
\begin{split}
\iint_{\mathcal{C}^+}\!\! & \frac{|w(x)-w(y)|^{p-2}(w(x)-w(y))}{|x-y|^{N+sp}}w(y)(\zeta(x)^p-\zeta(y)^p)\,dx\,dy\\
&\le \frac{2p}{R}\iint_{\mathcal{A}_1^+}\left\vert \frac{w(x)-w(y)}{|x-y|^{\frac{N}{p}+s}} \zeta(x)\right\vert^{p-1}\!\!\!\!\!\!\!\!
\frac{w(y)\,dx\,dy}{|x-y|^{\frac{N}{p}+s-1}}
+
\iint_{\mathcal{A}_2^+}\left\vert \frac{w(x)-w(y)}{|x-y|^{\frac{N}{p}+s}} \right\vert^{p-1}\!\!\!\!\!\!\!\!
w(y)\frac{\zeta(x)^p\,dx\,dy}{|x-y|^{\frac{N}{p}+s-1}}
\end{split}
\end{equation*}
We write the right hand-side in the form $ \mathcal{I}_1^+ +\mathcal{I}_2^+$ and
we make repeatedly use of Young inequality
\(
	pa^{p-1}b\le (p-1)\frac{a^p}{\tau^\frac{p}{p-1}}+\tau^pb^p\,,
\)	
with a suitable  $\tau>0$ to be determined. Estimating $\mathcal{I}_1^+ $ we get
\[
\begin{split}
\mathcal{I}_1^+ & \le
\tfrac{p-1}{\tau^\frac{p}{p-1}} \iint_{\mathcal{A}_1^+} \frac{|w(x)-w(y)|^p}{|x-y|^{N+sp}}\zeta(x)^p\,dx\,dy
+\frac{\tau^p}{R^p} \iint_{\mathcal{A}^+_1} \frac{w(y)^p}{|x-y|^{N+sp-p}}\,dx\,dy\\
& \le \tfrac{p-1}{\tau^\frac{p}{p-1}} \iint_{\mathcal{A}_1^+} \frac{|w(x)-w(y)|^p}{|x-y|^{N+sp}}\zeta(x)^p\,dx\,dy
+ \tau^p w(0)^p \omega_NR^{N-p} \int_0^R \rho^{(1-s)p-1}\,d\rho\\
& \le \tfrac{p-1}{\tau^\frac{p}{p-1}} \iint_{\mathcal{A}_1^+} \frac{|w(x)-w(y)|^p}{|x-y|^{N+sp}}\zeta(x)^p\,dx\,dy
+ \tau^p\tfrac{ \omega_N}{(1-s)p} w(0)^p R^{N-sp}\,.
\end{split}
\]
Similarly, 
\[
\begin{split}
\mathcal{I}_2^+ & \le
	\tfrac{p-1}{\tau^\frac{p}{p-1}} \iint_{\mathcal{A}_2^+} \frac{|w(x)-w(y)|^p}{|x-y|^{N+sp}}\zeta(x)^p\,dx\,dy
	+\tau^p \iint_{\mathcal{A}_2^+} \frac{w(y)^p\zeta(x)^p}{|x-y|^{N+sp}}\,dx\,dy\\
	& \le 	\tfrac{p-1}{\tau^\frac{p}{p-1}} \iint_{\mathcal{A}_2^+} \frac{|w(x)-w(y)|^p}{|x-y|^{N+sp}}\zeta(x)^p\,dx\,dy
	+\tau^p \tfrac{\omega_N}{(1-s)p} w(0)^pR^{N-sp}\,.
\end{split}
\]
Summing up gives
\begin{subequations}\label{colpm}
\begin{equation}\label{colpiu}
\begin{split}
\iint_{\mathcal{C}^+}\!\! &  \frac{|w(x)-w(y)|^{p-2}(w(x)-w(y))}{|x-y|^{N+sp}}w(y)(\zeta(x)^p-\zeta(y)^p)\,dx\,dy\\
& 
\le
\tfrac{p-1}{\tau^\frac{p}{p-1}} \iint_{\mathcal{C}^+} \frac{|w(x)-w(y)|^p}{|x-y|^{N+sp}}\zeta(x)^p\,dx\,dy
+ \tau^p C(N,s,p) w(0)^p R^{N-sp}\,.
\end{split}
\end{equation}
A similar argument also proves that
\begin{equation}
\label{colmeno}
\begin{split}
\iint_{\mathcal{C}^-}\!\! &  \frac{|w(x)-w(y)|^{p-2}(w(x)-w(y))}{|x-y|^{N+sp}}w(x)(\zeta(y)^p-\zeta(x)^p)\,dx\,dy\\
& \le	\tfrac{p-1}{\tau^\frac{p}{p-1}} \iint_{\mathcal{C}^-} \frac{|w(x)-w(y)|^p}{|x-y|^{N+sp}}\zeta(y)^p\,dx\,dy
+ \tau^p C(N,s,p) w(0)^p R^{N-sp}\,.
\end{split}
\end{equation}

\end{subequations}

We use the sum of \eqref{colpiu} and \eqref{colmeno} to estimate from above the sum of \eqref{5.2.1} and \eqref{5.2.2}.
In the inequality which we arrive at, the term divided by $\tau^\frac{p}{p-1}$ can be absorbed.
Taking into account \eqref{sottosol2},
it follows that there exist $C_1,C_2>0$, only depending on $N,s,p$, with
\begin{equation}\label{absorb}
	\int_{B_R} w\zeta^p\,dx  \ge (1-C_1)
	\iint_{\R^{2N} }\frac{|w(x)-w(y)|^p}{|x-y|^{N+sp}}\max\{\zeta(x),\zeta(y)\}^p\,dx\,dy- C_2 w(0)^p R^{N-sp}\,.
\end{equation} 
On the other hand,  by standard manipulations we also have
\begin{equation*}
	[w\zeta]_{s,p}^p\lesssim \iint_{\R^{2N}} 
	\!\!\!\!\!\!\frac{|w(x)-w(y)|^p}{|x-y|^{N+sp}}\max\{\zeta(x),\zeta(y)\}^p
	+\iint_{\mathcal{C}^+}\!\!\!\!w(y)^p \frac{(\zeta(x)-\zeta(y))^p}{|x-y|^{N+sp}}
	+\iint_{\mathcal{C}^-}\!\!\!\!w(x)^p \frac{(\zeta(y)-\zeta(x))^p}{|x-y|^{N+sp}}
\end{equation*}
where $\lesssim$ means $\le$ up to constants depending only on $p$. By \eqref{absorb}
we deduce
\begin{equation}\label{numeratore0}
		[w\zeta]_{s,p}^p
	\le C_3(N,s,p)\Big( w(0) R^N+w(0)^pR^{N-sp}
	+\mathcal{J}_+
	+\mathcal{J}_-
\Big)\,,
\end{equation}
where, thanks to the fact that $|\nabla\zeta|\le C R^{-1}$ and $0\le\zeta\le1$, we have
\begin{subequations}\label{Jstorto}
\begin{equation}
\begin{split}
	\mathcal{J}_+ & :=
	\iint_{\mathcal{C}^+}\!\!\!\!w(y)^p \frac{(\zeta(x)-\zeta(y))^p}{|x-y|^{N+sp}}\,dx\,dy\\
	& =\iint_{\mathcal{A}_1^+}\!\!\!\!w(y)^p \frac{(\zeta(x)-\zeta(y))^p}{|x-y|^{N+sp}}\,dx\,dy
	+\iint_{\mathcal{A}_2^+}\!\!\!\!w(y)^p \frac{\zeta(x)^p}{|x-y|^{N+sp}}\,dx\,dy \\
	& \le w(0)^p\left[\iint_{B_R\times B_R} \frac{R^{-p}dx\,dy}{|x-y|^{N-(1-s)p}}
	+ \iint_{B_R\times(\R^N\setminus B_R)} \frac{dx\,dy}{|x-y|^{N+sp}}\right] \le C w(0)^pR^{N-sp}\,,
\end{split}
\end{equation}
and similarly
\begin{equation}
	\mathcal{J}_+ := \iint_{\mathcal{C}^-}\!\!\!\!w(x)^p \frac{(\zeta(y)-\zeta(x))^p}{|x-y|^{N+sp}}\,dx\,dy
	\le C w(0)^p R^{n-sp}\,,
\end{equation}
\end{subequations}
with $C>0$ depending only on $N,s,p$. Combining \eqref{Jstorto} with \eqref{numeratore0}
we obtain
\begin{equation}
\label{numeratore}
\iint_{\R^{2N}} \frac{|w(x)\zeta(x)-w(y)\zeta(y)|^p}{|x-y|^{N+sp}}\,dx\,dy \le C_4(N,s,p) \Big( w(0)R^N+w(0)^pR^{N-sp}\Big)\,.
\end{equation}

To estimate the denominator in \eqref{lambdae}, we
recall
the notation introduced in~\cite{DKP}
\[
	{\rm Tail}(\varphi,x_0,r) = \left(r^{sp}\int_{\R^N\setminus B_r(x_0) }\frac{|\varphi(x_0)|^{p-1}}{
	|x-x_0|^{N+sp}}\,dx\right)^\frac{1}{p-1}
\]
for the {\em non-local tail} and we
 make use of the fact that for every $\delta>0$ we have
\begin{equation}
\label{brapar0}
\|w\|_{L^\infty(B_{R/4})} \le C_5(N,s,p) \left[ \left(\intmed_{B_{R/2}}w^p\,dx\right)^\frac{1}{p}+
\big(1+\delta\, {\rm Tail}(w,0,\tfrac{R}{4}) \big)R^\frac{sp}{p-1}
\right]\,,
\end{equation}
which follows by the estimate of~\cite[Theorem 3.8]{BP}, applied\footnote{In fact, that estimate implies \eqref{brapar0} with $\delta=1$,
but a close inspection of its proof at scale $1$ reveals that
minor arrangements allow for the interpolating parameter $\delta$ to appear.
} with $F\equiv1$.
Then, choosing $\delta=\delta_R$ so that $\delta\, {\rm Tail}(w,0,\tfrac{R}{4})\le1$, we obtain from \eqref{brapar0}
that
\begin{equation*}
\label{brapar}
\|u\|_{L^\infty(B_{R/4})} \le C_5(N,s,p) \left[ \left(\intmed_{B_{R/2}}w^q\,dx\right)^\frac{1}{q}+
2R^\frac{sp}{p-1}
\right]
\end{equation*}
where we also used Jensen inequality and the fact that $q\ge p$. The latter implies that
\begin{equation}
\label{braparappl}
\int_{B_{R/2}} w^q\,dx\ge \omega_NR^N \left( \frac{w(0)}{C_5} - 2R^\frac{sp}{p-1}\right)^q\,.
\end{equation}
Recalling that $\zeta\equiv1$ on $B_{R/2}$, with the choice $R = (w(0)/C_5)^\frac{p-1}{sp}$ 
inequality \eqref{braparappl} yields
\begin{equation}
\label{denominator}
\int_\Omega w^p\zeta^p\,dx \ge C_6(N,s,p,q) w(0)^{q+\frac{p-1}{sp}N}\,.
\end{equation}
Finally, combining \eqref{denominator} with \eqref{numeratore} we conclude by \eqref{lambdae} that
\(
	\lambda_{p,q}^s(\Omega)\le C_7(N,s,p,q)w(0)^{1-q}
\).
Since by assumption $w(0)=\|w\|_\infty $, we conclude.

To end the proof, we assume that $w:=w_{s,p,\Omega}$ belongs to $L^\infty(\Omega)$. Then
condition $\lambda_{p,p}^s(\Omega)>0$ plainly follows by the torsional Hardy inequality \eqref{coroineq}.
Indeed, we have
\[
	\int_\Omega |u|^p\,dx \le \|w\|_{L^\infty(\Omega)}^{p-1}\int_\Omega \frac{|u|^p}{w^{p-1}}\,dx
	\le \|w\|_{L^\infty(\Omega)}^{p-1}  \iint_{\R^{2N}} \frac{\big| u(x)-u(y)\big|^p}{|x-y|^{N+sp}}\,dx\,dy\,,
\]
for all $u\in C^\infty_0(\Omega)$, and in view of \eqref{lambda} with $q=p$ this gives the desired conclusion.
To deduce \eqref{superhom},
we observe that $\lambda_{p,p}^s(\Omega)>0$ implies $\lambda_{p,q}^s(\Omega)>0$ for $p<q<p^\ast_s$ as well,
by the Gagliardo-Nirenberg inequalities of Lemma~\ref{lm:GN},
and this concludes the proof. 
\qed

\subsection{Proof of Theorem~\ref{main2}}
 The proof is analogous to the one presented in \cite{BR} in the case $s=1$. By Theorem~\ref{main} (see in particular \eqref{subhom}) it suffices to show that
\[
	\lambda_{p,q}^s(\Omega)>0\Longleftrightarrow \text{$\mathcal{D}_0^{s,p}(\Omega)\hookrightarrow
		L^q(\Omega)$ is compact. }
\]
We prove the implication ``$\Longrightarrow$'', the other one being obvious by \eqref{lambda}.

We assume $\lambda_{p,q}^s(\Omega)>0$, and we abbreviate $w_{s,p,\Omega}$ to $w$. By
Theorem~\ref{main} (case $q<p$), we have
\begin{equation}\label{wgamma}
w\in L^{\frac{p-1}{p-q}q}(\Omega)\,.
\end{equation}
In addition, in view of Remark~\ref{cauchypoincare},
by \eqref{coroineq}, \eqref{poincare}, and the density of $C^\infty_0(\Omega)$ in $\mathcal{D}_0^{s,p}(\Omega)$ the assumption also implies  that
\begin{equation}
\label{coroineq2}
	\int_\Omega\frac{|u|^p}{w^{p-1}}\,dx
	\le \iint_{\R^{2N}} \frac{| u(x)-u(y)|^p}{|x-y|^{N+sp}}\,dx\,dy\,,
	\qquad
	\text{for all $u\in \mathcal D^{s,p}_0(\Omega)$,}
\end{equation}
and
\begin{equation}
\label{poincare2}
	\lambda_{p,q}^s(\Omega) \left(
		\int_\Omega|u|^q\,dx
	\right)^\frac{p}{q}
	\le
	\iint_{\R^{2N}} \frac{|u(x)-u(y)|^p}{|x-y|^{N+sp}}\,dx\,dy\,,\qquad
	\text{for all $u\in\mathcal D^{s,p}_0(\Omega)$.}
\end{equation}

Let $(u_n)_n$ be a bounded sequence in $\mathcal{D}_0^{s,p}(\Omega)$. The Gagliardo-Nirenberg
inequalities of Lemma~\ref{lm:GN} entail that the sequence is bounded in $L^p(\Omega)$, too.
Hence, possibly passing to a subsequence, we may assume that
$(u_n)_n$ converges weakly to a function $u$ in $\mathcal{D}_0^{s,p}(\Omega)$
and in $L^p(\Omega)$, since $p>1$ and both spaces are reflexive. Moreover, by \eqref{poincare2} the function
$u$ belongs to $L^q(\Omega)$. 

We prove that the sequence $v_n = u_n-u\in \mathcal{D}_0^{s,p}(\Omega)\cap L^p(\Omega)$ 
converges to $0$ strongly in $L^q(\Omega)$.
By Rellich-Kondra\v{s}ov theorem, this happens strongly in $L^q(\Omega\cap B_R)$, for all $R>0$. Hence, for every $R>0$ and 
for every $ \varepsilon>0$ there exists $n_{R, \varepsilon}\in \mathbb N$ with
\begin{equation}\label{dentro}
\int_{\Omega\cap B_R} |v_n|^q\,dx\le \varepsilon
\end{equation}
for all indices $n\ge n_{R, \varepsilon}$. If in addition, for every $\varepsilon$ there exists $R_\varepsilon>0$ such that 
\begin{equation}
\label{manca}
	\int_{\R^N\setminus B_{R_\varepsilon}} |v_n|^q\,dx\le C \varepsilon\,, \qquad \text{for all $n\in\mathbb N$,}
\end{equation}
for suitable a constant $C>0$ independent of $\varepsilon$ and $n$, 
then the sequence $(v_n)_n$  converges to $0$ strongly in $L^q(\Omega)$, as desired.

To prove \eqref{manca} we observe that, for every $R>1$,
by H\"older inequality we have
\begin{equation*}
	\int_{\Omega\setminus B_R} |v_n|^q\,dx \le \left( \int_\Omega \frac{|v_n|^p}{w^{p-1}}\,dx\right)^\frac{q}{p}
	\left( \int_{\Omega\setminus B_{R}} w^{\frac{p-1}{p-q}q}\,dx\right)^\frac{p-q}{q}\,.
\end{equation*}
Since the sequence $(v_n)_n$  is bounded in $\mathcal{D}_0^{s,p}(\Omega)$, by \eqref{coroineq2} the first
factor in the right hand member is bounded by a constant independent of $n$. As for the second one,
by \eqref{wgamma} the absolute continuity of the integral implies that for every $\varepsilon>0$ there exists
$R_\varepsilon>1$ with
\[
	\left( \int_{\Omega\setminus B_{R_\varepsilon-1}} w^{\frac{p-1}{p-q}q}\,dx\right)^\frac{p-q}{q}\le \varepsilon\,.
\]
The last two estimates entail \eqref{manca}, which concludes the proof.
\qed

\begin{ack}
The author wishes to thank Prof. Lorenzo Brasco for his useful comments on a preliminary version of the present manuscript, as well as  for suggesting the problem, in Osaka in May 2017, during the 
 Workshop ``Geometric Properties for Parabolic and Elliptic PDEs'', the organisers of which are also 
 gratefully acknowledged.
 This research is supported by the INdAM FOE 2014 grant ``SIES''.
\end{ack}


\begin{thebibliography}{99}
\bibitem{VDBB} M. van der Berg, D. Bucur. On the torsion function with Robin or Dirichlet boundary conditions. Journ. of Funct. Anal. {\bf 266 } (2014) 1647--1666.
\bibitem{BC}
L. Brasco, E. Cinti. On fractional Hardy inequalities on convex sets, preprint available at \url{http://cvgmt.sns.it/paper/3560/} (2017).
\bibitem{BF}
L. Brasco, G. Franzina. Convexity properties of Dirichlet integrals and Picone-type inequalities, Kodai Math. J., {\bf 37  }(2014), 769--799. 
\bibitem{BLP} L. Brasco, E. Lindgren, E. Parini. The fractional Cheeger problem, Interfaces Free Bound., {\bf 16} (2014), 419--458.
\bibitem{BMS}
L. Brasco, S. Mosconi, M. Squassina.
Optimal decay of extremals for the fractional Sobolev inequality,
Calc. Var. Partial Differential Equations 
{\bf 55} (2016), no. 2, Art. 23, 1-32.
\bibitem{BP}
L Brasco, E Parini.
The second eigenvalue of the fractional $p$-Laplacian,
Adv. in Calc. Var. {\bf 9} (4), 323-355.
\bibitem{BR} L. Brasco, B. Ruffini.
Compact Sobolev embeddings and torsion functions. Ann. Inst. H. Poincar\'e Anal. Non Lin\'eaire {\bf 34} (2017), no. 4, 817--843.
\bibitem{BB} D. Bucur, G. Buttazzo. On the characterization of the compact embedding of Sobolev spaces. Calc. Var.
Partial Differential Equations, {\bf 44 } (2012) 455-475.
\bibitem{DKP}
A Di Castro, T Kuusi, G Palatucci.
Local behavior of fractional $p$-minimizers,
Annales de l'Institut Henri Poincar\'e (C) Non Linear Analysis {\bf 33} (5), 1279-1299.
\bibitem{DL}
J. Deny, J.L. Lions. Les espaces du type de Beppo Levi. Ann. Inst. Fourier {\bf 5} (1954) 305--370.
\bibitem{HL}
L. H\"ormander, J. L. Lions. Sur la compl\'etion par rapport \`a une int\'egrale de Dirichlet.
{\em Math. Scand.} {\bf 4} (1956), 259--270.
\bibitem{IMS}
A. Iannizzotto, S. Mosconi, M. Squassina.
Global H\"older regularity for the fractional p-Laplacian.
Rev. Mat. Iberoam. {\bf 32} (2016), no. 4, 1353--1392.
\bibitem{LL}
E. Lindgren, P. Lindqvist. Fractional eigenvalues. Calc. Var. Partial Differential Equations {\bf 49} (2014), 795--826.
\bibitem{Mb}
V. Maz'ya, Sobolev spaces, Sobolev spaces with applications to elliptic partial differential equations. Grundlehren der Mathematischen Wissenschaften, 342. Springer, Heidelberg, 2011.
\bibitem{MS}
V. Maz'ya, T. Shaposhnikova. On the Bourgain, Brezis, and Mironescu theorem concerning limiting embeddings of fractional Sobolev spaces, J. Funct. Anal., {\bf 195} (2002), 230--238.
\end{thebibliography}
\end{document}